\documentclass[11pt]{amsart}[draft]
\usepackage{graphicx, overpic}
\usepackage[below]{placeins}
\usepackage[colorlinks=true, linkcolor=blue, citecolor=blue]{hyperref}
\usepackage[]{algorithm2e}

\usepackage{graphicx} 
\usepackage{tikz} 
\usepackage{tikz-cd} 

\usepackage{thmtools}
\usepackage{thm-restate}
\usepackage{hyperref}

\usepackage{cleveref}

\usepackage[T1]{fontenc}

\usepackage{amsmath,amsthm,amscd,amssymb,eucal}

\usepackage{enumerate, amsfonts, latexsym, color, url}
\usepackage{epstopdf}

\usepackage{pinlabel}

\makeatletter
\newsavebox{\@brx}
\newcommand{\llangle}[1][]{\savebox{\@brx}{\(\m@th{#1\langle}\)}%
  \mathopen{\copy\@brx\kern-0.5\wd\@brx\usebox{\@brx}}}
\newcommand{\rrangle}[1][]{\savebox{\@brx}{\(\m@th{#1\rangle}\)}%
  \mathclose{\copy\@brx\kern-0.5\wd\@brx\usebox{\@brx}}}
\makeatother

\begin{document}

\newtheorem{theorem}{Theorem}[section]
\newtheorem{lemma}[theorem]{Lemma}
\newtheorem{proposition}[theorem]{Proposition}
\newtheorem{corollary}[theorem]{Corollary}
\newtheorem{conjecture}[theorem]{Conjecture}
\newtheorem{question}[theorem]{Question}
\newtheorem{problem}[theorem]{Problem}
\newtheorem*{claim}{Claim}
\newtheorem*{criterion}{Criterion}
\newtheorem*{universal_circle_theorem}{Universal Circle Theorem~\ref{theorem:universal_circle}}
\newtheorem*{quasigeodesic_theorem}{Quasigeodesic Flow Theorem~\ref{theorem:zippers_from_flows}}
\newtheorem*{quasimorphism_theorem}{Uniform Quasimorphism Theorem~\ref{theorem:zippers_from_quasimorphisms}}
\newtheorem*{order_theorem}{Uniform Order Theorem~\ref{theorem:zippers_from_orders}}

\theoremstyle{definition}
\newtheorem{definition}[theorem]{Definition}
\newtheorem{construction}[theorem]{Construction}
\newtheorem{notation}[theorem]{Notation}
\newtheorem{object}[theorem]{Object}
\newtheorem{operation}[theorem]{Operation}

\theoremstyle{remark}
\newtheorem{remark}[theorem]{Remark}
\newtheorem{example}[theorem]{Example}

\numberwithin{equation}{subsection}

\newcommand\id{\textnormal{id}}

\newcommand\N{\mathbb N}
\newcommand\Z{\mathbb Z}
\newcommand\R{\mathbb R}
\newcommand\C{\mathbb C}
\newcommand\EE{\mathcal E}
\renewcommand\H{\mathbb H}
\newcommand\A{\mathcal A}
\newcommand\F{\mathcal F}
\newcommand\G{\mathcal G}
\newcommand\CC{\mathcal C}
\newcommand\QQ{\mathcal Q}
\newcommand\Sp{\textnormal{Sp}}
\newcommand\SL{\textnormal{SL}}
\newcommand\Homeo{\textnormal{Homeo}}
\newcommand\inte{\textnormal{int}}
\newcommand\un{\textnormal{univ}}
\newcommand{\red}{\textcolor{red}}

\newcommand{\Lp}{\Lambda^{+}}
\newcommand{\Lm}{\Lambda^{-}}
\newcommand{\Lpm}{\Lambda^{\pm}}
\newcommand{\deffont}[1]{\emph{#1}}

\title{Infinite chains of perfect fits for Expanding Thurston maps}

\author{Ino Loukidou}
\address{University of Chicago \\ Chicago, Ill 60637 USA}
\email{thelouk@uchicago.edu}
\date{\today}

\begin{abstract}
   
The topological mating of two postcritically finite polynomials with dendritic
Julia sets is encoded by a pair of circle laminations $\Lambda^{\pm}$, 
whose collapse produces a sphere-filling curve. When a leaf of
$\Lambda^{+}$ and a leaf of $\Lambda^{-}$ share an endpoint they form a \emph{perfect fit}.
An unpublished proposition of Epstein,
recorded by Petersen and Meyer, shows that for matings of honest degree-$d$
polynomials an infinite-diameter ray equivalence class, i.e. an infinite chain of perfect fits is impossible. 
We show that this finiteness is a genuinely holomorphic
phenomenon. Allowing the dynamics to carry a periodic critical orbit, 
we construct \emph{combinatorially} expanding Thurston maps---realized by no expanding rational map---that admit invariant sphere-filling curves 
yet whose laminations $\Lambda^{\pm}$ contain infinite chains of perfect fits, in
fact infinitely many of them. 

\end{abstract}

\maketitle
\setcounter{tocdepth}{1}
\tableofcontents

\newcommand{\crit}{\operatorname{crit}}
\newcommand{\post}{\operatorname{post}}

\section{Introduction}

\emph{Mating}, introduced by Douady and Hubbard \cite{Douady}, combines two
polynomials of the same degree $d$ into a single dynamical system on the sphere.
In its most transparent form it applies to two postcritically finite polynomials
$f_{+},f_{-}$ whose Julia sets $J_{\pm}$ are \emph{dendrites}. Each $J_{\pm}$ is
a quotient of the closed disk by a lamination; equivalently, the action of
$f_{\pm}$ on $J_{\pm}$ is a factor of the map $h\colon z\mapsto z^{d}$ of the
circle under an invariant equivalence relation, or \emph{laminar relation},
$\Lambda^{\pm}$ of $S^{1}$. To \emph{mate} $J_{+}$ and $J_{-}$ is to glue the two
disks along their boundary circles by $z\mapsto\bar z$ and to collapse both
laminations at once, i.e.\ to form the quotient of $S^{1}$ by the equivalence
relation generated by $\Lp$ and $\Lm$ together. When the quotient is a
topological sphere one says that $J_{+}$ and $J_{-}$ have been
\emph{topologically mated}.

The relation generated by $\Lp$ and $\Lm$ can be intricate. A nontrivial pair of
points of $S^{1}$ identified by $\Lambda^{\pm}$ is a \emph{leaf}; when a leaf of
$\Lp$ and a leaf of $\Lm$ share an endpoint they form a \emph{perfect fit}. A
sequence of leaves $\ell_{1},\ell_{2},\dots,\ell_{m}$, drawn alternately from
$\Lp$ and $\Lm$, in which one endpoint of $\ell_{i}$ is an endpoint of
$\ell_{i+1}$, is a \emph{chain of perfect fits}; the \emph{diameter} of an
equivalence class is the length of the longest chain of perfect fits it
contains, and the mating has finite diameter if these lengths are uniformly
bounded. These classes and their diameters are exactly the
\emph{ray-equivalence classes} of the polynomial-mating
literature \cite{Contradiction, Jung}: a leaf records two external rays landing at a common
point of a Julia set, a perfect fit records a point of $J_{+}$ glued to a point
of $J_{-}$, and a chain of perfect fits is a path of extended external rays that
alternates between the two filled Julia sets.

Whether a topological mating exists---whether the quotient is a sphere---is
controlled by these classes. An unpublished proposition of A.~Epstein, recorded
by Petersen and Meyer \cite[Prop.~4.12]{Contradiction}, states that if some
ray-equivalence class is a cycle or has infinite diameter, then the quotient of
the (formal) mating is not homeomorphic to $S^{2}$; in other words, the
topological mating fails to exist. On the other hand, when the mating comes from
an \emph{expanding} Thurston map, Meyer \cite{ExpandingMaps,Quotients} shows that the
collapse of $\Lpm$ \emph{is} a genuine sphere-filling curve
$\gamma\colon S^{1}\to S^{2}$, so the quotient is a sphere. Taken together, the
two results say that for matings of honest degree-$d$ dendritic Julia sets---the
Misiurewicz case, equivalently postcritically finite \emph{expanding rational}
maps---every chain of perfect fits is finite: infinite chains are forbidden by
the holomorphic structure.

This raises the question of whether finiteness is intrinsic to the combinatorial
construction---collapsing a pair of $h$-invariant laminations to a sphere-filling
curve---or a feature of the holomorphic setting in which Epstein's proposition is
proved. Petersen and Meyer establish Proposition 4.12 for honest polynomials,
using their external rays and Carath\'eodory loops, which prompts:

\begin{question}\label{mainquestion}
Does Proposition 4.12 remain true when the map $f$ is not rational? That is, must
the ray-equivalence classes of a non-rational expanding Thurston map still have
finite diameter?
\end{question}

We note that examples of matings that produce infinite ray equivalence classes are known to exist in the literature. In \cite[Section 5]{Jung}, Jung gives examples of matings of degree 2 polynomials which, when mated, produce an infinite ray equivalence class. The difference there is that in his examples the equivalence classes are cyclic and not closed, giving a quotient that is not homeomorphic to the sphere.

\subsection{Statement of results}

We answer Question \ref{mainquestion} in the negative: finiteness of chains of perfect fits is
special to the holomorphic setting.

\begin{theorem}\label{maintheorem}
There exist combinatorially expanding Thurston maps $f\colon S^{2}\to S^{2}$ that
admit an $f$-invariant sphere-filling curve $\gamma\colon S^{1}\to S^{2}$, given
by collapsing a pair of $h$-invariant laminations $\Lpm$ of $S^{1}$ that contain
infinite chains of perfect fits.
\end{theorem}

There is no conflict with Epstein's proposition, and seeing why is the point of
the construction. The first thing to say is that $S^{2}$ is not an abstract
quotient we must certify but the ambient sphere on which $f$ acts: the curve
$\gamma$ is built directly on it, as the limit of the approximating Eulerian
circuits $\gamma_{n}$ drawn on the $n$-tiles, and $\Lpm$ merely records its
fibers, so $S^{1}/{\sim}\;\cong\;S^{2}$ is automatic. Epstein's proposition, by
contrast, asks whether such a $\gamma$ can exist at all for honest polynomials,
and there an infinite chain of perfect fits is fatal---for reasons (honest
external rays and Carath\'eodory loops) that have no counterpart in our
combinatorial setting. What remains is to explain why our maps escape the
holomorphic finiteness. A postcritically finite polynomial with dendritic Julia
set is necessarily \emph{Misiurewicz}: all of its critical points are strictly
preperiodic. A \emph{periodic} critical point is impossible for a dendrite, since
it would create an attracting cycle and hence a nonempty Fatou set, whereas a
dendrite has empty interior. Our examples are dendritic only \emph{topologically}.
We allow the dynamics to carry an honest periodic critical orbit: the maps remain
postcritically finite branched covers---in our main example
$|\post(f)|=3$, so $f$ can even be realized by a rational map on $\mathbb{C} \mathbb{P}^{1}$---but
by no \emph{expanding} one, because a periodic critical point obstructs metric
expansion \cite{Book}. This is precisely why our maps are only
\emph{combinatorially} expanding, and precisely what places them outside the
hypotheses of Proposition 4.12.

The object being mated is then not an honest Julia set but a \emph{generalized
dendrite}: the topological quotient of the filled Julia set $K$ obtained by
collapsing, roughly, the iterated preimages of an invariant Hubbard tree. Away
from the critical orbit this looks like an ordinary dendrite, but at a periodic
critical point it is \emph{infinitely-valent}---its complement has infinitely
many components---because the ramification of $f^{n}$ at that point grows without
bound, accumulating infinitely many local branches. Equivalently, the
corresponding equivalence class of $\Lpm$ in $S^{1}$ contains infinitely many
points. For an honest dendrite this cannot occur: every fiber of the
Carath\'eodory semi-conjugacy is finite \cite[\S3.4]{Contradiction}, so every class is
finite, and hence so is its diameter. Infinite valence at a periodic critical
point is exactly the source of the infinite chains of perfect fits.

The mechanism is captured by a simple sufficient condition, deduced from a
description (Lemma \ref{lemma}) of how the decomposition element $[v]_{n}$ at a critical
point $v$ grows under pullback.


\begin{restatable}{cor}{CorAB}\label{a and b}
    Suppose $f$ has a critical cycle $C$ containing critical points $a,b$ with
$[a]_{1}$ a leaf of $\Lp$ and $[b]_{1}$ a leaf of $\Lm$. Then the equivalence
classes $[a],[b]$ contain infinite chains of perfect fits.
\end{restatable}

Along the cycle each pullback attaches a fresh copy of the opposite-colored leaf
to the current decomposition element (Lemma \ref{lemma}), so $[a]_{n}$ contains an
alternating chain of length at least $n$; in the limit $[a]$ contains an infinite
chain. Since every iterated preimage of such a point is again infinitely-valent,
in fact infinitely many equivalence classes contain infinite chains.

In Section \ref{example} we make this explicit (Figure \ref{The map}): a degree-$8$ combinatorially
expanding Thurston map $f$ with a period-$2$ cycle $\{1,\infty\}$ of two
degree-$2$ critical points, arranged so that $[1]_{1}$ is a single leaf of $\Lp$
and $[\infty]_{1}$ is a single leaf of $\Lm$ (Figures \ref{The curve}, \ref{1 and infty}). Corollary \ref{a and b} then
shows that $[1]$ and $[\infty]$ contain infinite chains of perfect fits; the
first five approximations of $[\infty]$ are drawn in Figure \ref{5 steps}.

\subsection{Context: sphere-filling curves}

The pair of laminations above is one instance of a much broader phenomenon.
Sphere-filling curves arise throughout dynamics and geometry as collapses of a
pair of laminations of $S^{1}$. The classical example is the Cannon--Thurston
map \cite{cannonthurston}: for a hyperbolic $3$-manifold fibering over the circle with
fiber $\Sigma$ and pseudo-Anosov monodromy, the stable and unstable laminations
lift to a pair $\Lambda^{\pm}$ of laminations of
$S^{1}=\partial\mathbb{H}^{2}$, and the induced quotient $S^{1}/\!\sim$ is a
sphere with quotient map a sphere-filling curve $\gamma$. Adding an entry to
Sullivan's dictionary, Meyer \cite{ExpandingMaps} showed that an expanding Thurston map
admitting an invariant Jordan curve carries an invariant sphere-filling curve
$\gamma\colon S^{1}\to S^{2}$, and \cite{Quotients} that $\gamma$ is again the
collapse of a pair of laminations $\Lpm$ of $S^{1}$---exactly the two laminations
one obtains by \emph{unmating} $f$ \cite{UnmatingMeyer}. When $f$ is an expanding
rational map the two pieces are honest polynomials with dendritic Julia sets;
in general they are only topological, as in this paper.

In the context of Sullivan's dictionary, it's also worth noting that finiteness and geometry are shown to be related in both the hyperbolic 3-manifolds and the expanding Thurston maps world.
In \cite{Fenley}, Fenley shows \emph{finiteness of Lozenges} for quasigeodesic pseudo-Anosov flows in hyperbolic 3-manifolds. In our language, this translates to chains of perfect fits in the orbit space of the flow having finite length, giving a clear analogy to \cite[Prop.~4.12]{Contradiction} about holomorphic expanding Thurston maps.

In \cite{CaTherineWheels}, Calegari and the author study a class of such curves, the
\emph{CaTherine wheels}, that arise in the same way. It is shown there that
Meyer's curves are CaTherine wheels precisely when the corresponding laminations
have \emph{no} perfect fits, and it is conjectured that in the presence of
perfect fits one obtains a broader class, the \emph{P-CaTherine wheels}, for which
most of the theory persists. The present examples---sphere-filling curves whose
laminations carry not merely perfect fits but infinite chains of them---lie at
the far end of this range.

\subsection{Outline of the paper} The paper is organized as follows.
Section \ref{background} covers the preliminaries on Expanding Thurston maps, connections and laminations of the circle. 
In section \ref{section3}, we discuss the process of \emph{unmating} of a Thurston map as described by Meyer in \cite{UnmatingMeyer}. To illustrate the process we introduce the example of the barycentric subdivision map (figure \ref{barycentric}), that appears in \cite[Chapter 12]{Book}.
In section \ref{pullback}, we prove the basic lemmas for describing decomposition elements and prove Corollary \ref{a and b}.
Finally, in section \ref{example}, we give an example of an Expanding Thurston map that, when unmated, it produces a pair of laminations that when collapsed they have a spherical quotient but they contain infinite chains of perfect fits, proving theorem \ref{maintheorem}.

\subsection{Acknowledgements}

I am deeply grateful to Danny Calegari for bringing this question to my attention, as well as for countless helpful discussions and his constant support. 
I would also like to thank Daniel Meyer for comments on a previous draft and for bringing to my attention further examples and different cases of matings of polynomials and
Lucas Kerbs for conversations about decomposition elements and laminations.

\section{Background}\label{background}

\subsection{Branched coverings of the sphere.}

Expanding Thurston maps are examples of branched covering maps of the sphere to itself.
Let $f:S^2 \to S^2$ be a branched covering.  
The degree $d$ of the map $f$ is its degree as a regular covering map away from the branch locus.
The set of \textit{critical points} of $f$, $\operatorname{crit}(f)$, is defined to be the branch locus of $f$ and the \textit{degree} of a critical point $c \in \crit(f)$ is the local degree of $f$ at $c$.
In particular, by the degree-genus formula we have that $$2+ \sum_{c \in \crit(f)}\left( d_f(c) - 1 \right ) = 2d$$
A \textit{postcritical point} of $f$ is an iterated image of a critical point of $f$. In particular,
$$\post (f): = \bigcup_{n \geq 1}\left \{ f^n(c) \ | \ c\in \crit(f) \right \}$$

A map $f$ is called \textit{postcritically finite} if the set $\post(f)$ is finite.
Equivalently, $f$ is postcritically finite if all of its critical points are preperiodic.

\begin{definition}
    A \textit{Thurston map} is a branched covering map $f:S^2\to S^2$ of degree $d \geq 2$ that is postcritically finite.
\end{definition}

There are no Thurston maps with $|\post| \leq 1$, and the only ones with $|\post| = 2$ are equivalent to $z \mapsto z^d$ for some $d$, thus our focus lies on maps with $|\post(f)|\geq 3$.

For an exposition to such maps, see \cite{Book}. The theorems that follow in this section are all stated and proved there.

\begin{definition}
    A Thurston map is called \textit{expanding} if there exists a Jordan curve $\mathcal{C} \subseteq S^2$ with $\post(f) \subseteq \mathcal{C}$ and the maximal diameter of a component of $S^2 \smallsetminus f^{-n}(\mathcal{C})$ goes to 0 as $n \to \infty$.

\end{definition}

Suppose that $f:S^2 \to S^2$ has $|\post(f)| = n$. To describe $f$ we will use a \textit{two-tile subdivision rule}.
\begin{definition}
    A two-tile subdivision rule for $S^2$ is a triple $(D_1,D_0,L)$ of cell decompositions $D_0,D_1$ of $S^2$ and an orientation-preserving labeling $L:D_1\to D_0$, such that:
    \begin{enumerate}
        \item $D_0$ contains precisely two tiles (a \textit{white} and a \textit{black} one).
        \item $D_1$ is a refinement of $D_0$ and $D_1$ contains more than two tiles.
        \item If $k$ is the number of vertices in $D_0$, then $k \geq 3$ and every tile in $D_1$ is a $k$-gon.
        \item Every vertex in $D_1$ is contained in an even number of tiles in $D_1$.
    \end{enumerate}
\end{definition}
We will refer to the tiles of $D^1$ as white or black 1-tiles, depending on their image in $D_0$ under $L$.
 
We say that $f$ \textit{realizes} the two tile subdivision rule if $f$ is cellular for $(D_0,D_1)$ and $f(\tau) = L(\tau)$ for each $\tau\in D_1$.

Let $J$ be the common boundary of the two 0-tiles.

\begin{proposition}
    Suppose $(D_0,D_1,L)$ is a two-tile subdivision rule on $S^2$. Then there exists a Thurston map $f:S^2\to S^2$ that realizes $(D_0,D_1,L)$. 
    Moreover, the Jordan curve $J$ is $f$-invariant and contains the set $\post(f)$.
\end{proposition}

In particular, in our examples the set of vertices of the tiles in $D_0$ is exactly the set $\post(f)$.

 \subsubsection{Connections at a point}\label{connections}

    Suppose that the map $f:S^2 \to S^2$ realizes the two-tile subdivision rule $(D_0,D_1,L)$.
    The union of the boundaries of the 1-tiles form a planar graph $\Gamma$ that can be two-colored and thus it admits an Eulerian circuit.
    In \cite{ExpandingMaps}, Meyer constructs a sphere-filling curve that is invariant under $f$.
    The construction is described inductively in steps.
    The first step is to choose an Eulerian circuit $\gamma$ that is the end of a \textit{pseudo-isotopy} that deforms the invariant curve $J$ into it (\cite{ExpandingMaps}, definition 3.2).
    In particular, this condition implies that $\gamma$ does not \textit{cross} itself and that it contains all white 1-tiles on one of its sides and all black ones on the other.

    This is equivalent to choosing a \textit{connection} for each vertex of the graph $\Gamma$.
    In particular, let $c \in \crit(f)$. If $\deg_f(c) = k$, $k$ pairs of black and white tiles meet at $c$.
    Keeping track of the way $\gamma$ traces the edges around $c$, we get a complementary non-crossing partition, in the sense of \cite{Quotients} and \cite{ExpandingMaps}, like in figure \ref{connections figure}.
    We will be referring the this partition as the \textit{connection} of $f$ at $c$.
    Note that if we want to be precise and draw exactly how $\gamma$ looks like, we will be drawing the graph of edges of $D_1$, which gives no information about the way the circuit traces it.
    Therefore, for the sake of visualization, we will be pulling our circuits off of themselves at a neighborhood of the vertices.
    \begin{figure}[h!]
        \centering
        \includegraphics[width=0.6\linewidth]{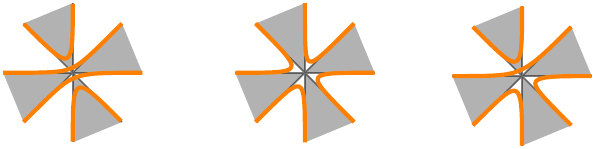}
        \caption{Examples of connections at a degree 4 critical point.}
        \label{connections figure}
    \end{figure}


    \subsubsection{Iterating $f$}

    Suppose $f$ is given by a two tile subdivision rule. Then, it is easy to see how $f^2$ behaves. 
    To each black/white 1-tile we can apply the subdivision rule again to obtain the black/white 2-tiles.
    Inductively, we obtain black/white $n$-tiles by subdividing the $(n-1)$-tiles by our rule.

    In the sequel we will see how we can obtain \textit{pullbacks} of the circuit $\gamma$ to create a sequence of Eulerian circuits $\gamma_n$ for the edge graph of the $n$-tiles (the $n$-th approximations of the invariant curve of \cite{ExpandingMaps}).
    Equivalently, we will describe how to choose a connection of n-tiles in a compatible way.
    For now, we can confidently say that if $\gamma_2$ were to be the pullback of $\gamma =: \gamma_1$, then the following should hold:
    If $f(b) = a \in \crit(f)$ is a critical point of degree $k$ and $b$ is not a critical point of $f$, then $b$ is a degree $k$ critical point of $f^2$ and the 2-tiles around $b$ are mapped by $f$ to the 1-tiles around $a$. 
    Since we already have a connection at $a$ for $\gamma_1$, we can imagine to simply ``pull it back'' to get a connection at $b$ for $\gamma_2$.
    On the other hand, choosing the connection for $\gamma_2$ at a critical point of $f^2$ that was already a critical point of $f$ can be more tricky, but we will see later how it is possible to make a consistent choice.

\subsection{Laminations} In this section we give some background on laminations of the circle $S^1$ that are invariant under a covering map $t \mapsto d\cdot t$ and the decompositions they define.

\begin{definition}
    A \textit{lamination} $\Lambda$ of $S^1$ is a closed subset of the space of unordered pairs of points in $S^1$ such that no two of its elements are linked in $S^1$.
\end{definition}

A \textit{leaf} $\lambda$ of a lamination $\Lambda$ is a pair $\{p,q\} \in \Lambda$. 
We wish to later use a pair of laminations of $S^1$ as data for identifying its points. 
Under the quotient map, points of $S^1$ that are endpoints of the same leaf are identified. Consequently, endpoints of consecutive leaves---leaves that share an endpoint with each other---are also identified, as well limits of sequences of endpoints of consecutive leaves.
Therefore we will be referring to the closures of sequences of consecutive leaves as \textit{decomposition elements} (see section \ref{decompositions} for details).
Note that if the laminations have finitely many leaves, there is no need to take any closures.

Sometimes it is convenient to alternate between thinking of a leaf as an unordered pair of points in $S^1$ or as a geodesic in $\H^2$, where $S^1 = \partial \H^2$.
In that case, instead of considering decomposition elements in $S^1$, it will be useful to consider their convex hulls in $\H^2$---that is, ideal polygons with vertices at the endpoints of the leaves. By abuse of notation, we will also call them decomposition elements.

    \subsubsection{Degree $d$ majors}
    In this section, we think of the circle $S^1$ as $\R / \Z$. We fix an integer $d>0$ and consider the map $h:S^1\to S^1$, $x \mapsto d\cdot x$.

    \begin{definition}
        A degree $d$ \textit{critical leaf} of multiplicity $n$ is an unordered $(n+1)$-tuple of distinct points $\{x_0,\dots,x_n\}$ for which $d(x_i-x_j) = 0 \mod \Z$ for all $i,j$.

        A (degree $d$) \textit{major} is a collection $C$ of disjoint degree $d$ critical leaves which are pairwise disjoint and have total multiplicity $d-1$.
    \end{definition}

    \begin{figure}[h!]
        \centering
        \includegraphics[width=0.7\linewidth]{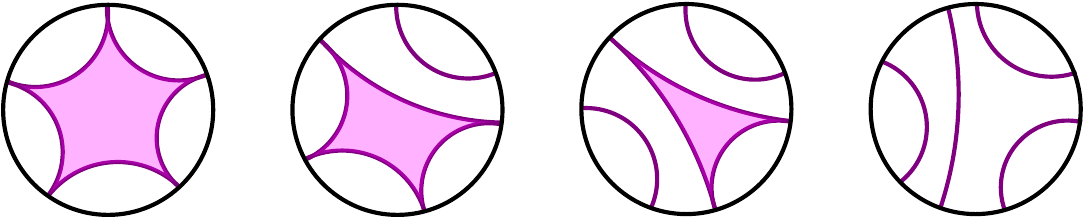}
        \caption{Examples of degree 5 majors. Note that all endpoints of the same leaf differ by a multiple of $e^{\frac{2 \pi i}5}$.}
        \label{fig:placeholder}
    \end{figure}

    A critical leaf is simple if it has multiplicity 1, and a major is simple if all its critical leaves are simple.
    A non-simple critical leaf of multiplicity $n$ may be thought of as a union of $n$ simple leaves, and a major may be thought of as a finite lamination. 

    If $\lambda: = \{x,y\}$ is a non-critical leaf, we define $h\lambda: = \{d \cdot x,d \cdot y\}$, and otherwise we define $h\lambda$ to be empty.

    \begin{definition}
        A lamination $\Lambda$ of $S^1$ is said to be $h$-\textit{invariant} if
        \begin{enumerate}
            \item for every non-critical leaf $\lambda$, the leaf $h\lambda$ is a leaf of $\Lambda$; and
            \item for every leaf $\lambda$ there are exactly $d$ leaves $\mu$ of $\Lambda$ with $h\mu = \lambda$, and these leaves are all disjoint.
        \end{enumerate}
    \end{definition}


    \begin{proposition}\label{pullback majors}
        For any $d>1$ and a degree $d$ major $C$ there is an invariant lamination $\Lambda$ containing $C$.

        Furthermore, suppose that there are no points $x,y\in S^1$ (not necessarily distinct) contained in leaves of $C$ and an $n$ such that $h^n(x) = y$. Then $\Lambda$ is unique.
    \end{proposition}

    A proof of this proposition can be found in \cite{WhatsNext}, Theorem 4.3 or a sketch of it at \cite{CaTherineWheels}, Proposition 7.3.
    The proof is constructive, and the invariant lamination $\Lambda$ is obtained as a limit of finite approximations $\Lambda^{\pm}_n$. 
    The main idea is that to get $\Lambda_n$, we start with the major $C:= \Lambda_1$ and take iterated preimages of its leaves under the map $h$.

    Actually, later we will start with a \textit{pair} of degree $d$ majors $C^{\pm}$. 
    Following the previous procedure, they give rise to a sequence of pairs of finite laminations $\Lambda^{\pm}_n$ and consequently to a pair of $h$-invariant laminations $\Lambda^{\pm}$.

    In section \ref{mating}, we illustrate the approximations $\Lambda^{\pm}_1$ and $\Lambda^{\pm}_2$ for an example (figure \ref{1st approximation}, down right and figure \ref{2nd approximation}, left).
 

    \section{From connections to laminations}\label{section3}

    In the previous sections we talked about the approximations $\gamma_n$ of an $f$-invariant sphere-filling curve and the approximations $\Lambda^{\pm}_n$ to a pair of $h$-invariant laminations. 
    The rest of this paper, is dedicated to discussing connecions between those approximations.

    \subsection{Connections and Decomposition elements}\label{connections and decompositions}
    Consider the first approximation $\gamma_1$ of the $f$-invariant sphere filling curve $\gamma$, and let $c\in \crit(f)$ be a critical point.
    As discussed in section \ref{connections}, $\gamma_1$ is a loop that has two sides, a white and a black one, and bumps into itself at each critical point $c$, in particular, it passes through $c$ exactly $k = \deg_f(c)$ times.
    
    It is convenient now to see the loop $\gamma_1$ as a parametrization $\gamma_1:S^1 \to E$, where $E$ is the set of edges of $D_1$. 
    Every time $\gamma_1$ bumps into itself at $c$, i.e. there are $s,t \in S^1$ such that $\gamma_1(s) = \gamma_1(t) = c$, one of the two sides degenerates to a point.
    Draw a red leaf on $S^1$ joining $s$ and $t$ every time the white side vanishes and a blue leaf if the black side vanishes.
    This way, we end up having a collection of $k-1$ red and blue leaves corresponding to each critical point of order $k$, precisely described by the choice of connection at that point, like in figure \ref{Connections Laminations}. 
    These leaves combined make up a decomposition element for $\Lambda^{\pm}_1$, which we will be referring to as $[c]_1$.

    \begin{figure}[h!]
        \centering
        \includegraphics[width=0.56\linewidth]{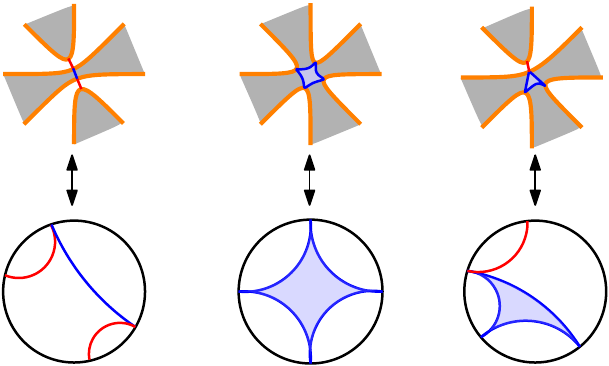}
        \caption{The choice of connection at a point precisely describes a collection of leaves and vice versa.}
        \label{Connections Laminations}
    \end{figure}

    Reversing the previous process, given a decomposition element consisting of $k-1$ leaves (a polygon of $m$ sides counts as a collection of $m-1$ leaves), we can draw the associated connection at the point $c$, also shown in figure \ref{Connections Laminations}.

     \subsection{Mating and unmating of polynomials}\label{mating}

     As stated in the introduction, the operation of mating of polynomials was introduced by Douady and Hubbard in \cite{Douady}. It is a way to combine two polynomials to create a rational map $\C \mathbb{P}^1 \to \C \mathbb{P}^1$. 
     Later, Meyer (\cite{Quotients}, \cite{UnmatingMeyer}, \cite{ExpandingMaps}) proves the opposite. Namely, given a rational map it is possible to unmate it to retrieve the data of two polynomials.
     This process actually works for general Expanding Thurston maps with some additional properties and uses strongly the machinery of laminations and decompositions of the sphere.
     The difference here is that, if the map is not rational, the unmating process may result to objects that are just topological, and don't correspond to polynomials with Julia set the whole sphere.
     In this section, we give a short exposition to the process of unmating an expanding Thurston map. 
     Proofs of all the claims can be found in \cite{Quotients}, sections 5 and 6.

    To illustrate the process we will use as running example the barycentric subdivision map $f$, described in figure \ref{barycentric}. 
    For simplicity, we will represent the sphere $S^2$ as a triangular pillow case, and many times we will add more copies of a fundamental domain, as appears in figure \ref{sphere}.
    \begin{figure}[h!]
        \centering
        \includegraphics[width=0.5\linewidth]{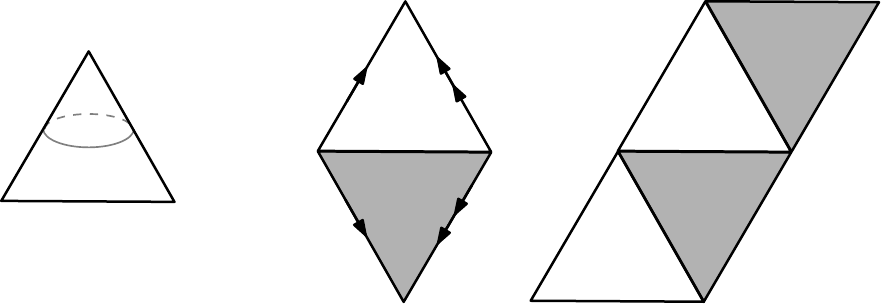}
        \caption{Depictions of the sphere $S^2$ as a gluing of two triangles. (Right) Part of the developing map.}
        \label{sphere}
    \end{figure}

    The map $f$ is discussed in \cite{Book}, chapter 12. It is a degree 6 combinatorially expanding Thurston map with $|\operatorname{post}(f)| = 3$. Even though $f$ can be represented by a holomorphic map, this can't be expanding since it contains a periodic critical point.
    \begin{figure}[h!]
        \centering
        \includegraphics[width=0.45\linewidth]{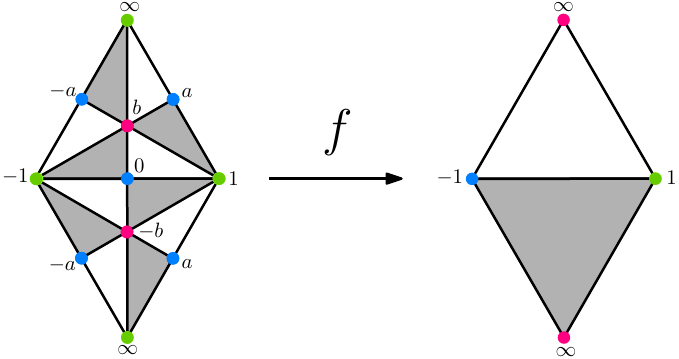}
        \caption{The barycentric subdivision map. The coloring of the vertices on the left indicates the point they map to under $f$.}
        \label{barycentric}
    \end{figure}

    We will now describe the process of unmating $f$ to obtain a pair of degree $d$ majors.

     \begin{enumerate}
         
         \item Draw the first approximation of an invariant Peano curve, as described in section \ref{connections}.
         Note that this curve may not be unique. In our example, we will work with the one that appears in figure \ref{1st approximation} (up left).

         \item Draw the two (finite) laminations implied by this approximation, as described in section \ref{connections and decompositions} (Figure \ref{1st approximation}, up right).
        
         \item Straighten the circle to get the combinatorial data of a pair of finite laminations of $S^1$. These will give us the appropriate pair of majors, and the leaves will be the critical leaves. (Figure \ref{1st approximation}, down left).

    \end{enumerate}

        Next, we're going to get \textit{coordinates} for the leaves we obtained.

    \begin{enumerate}

         \item [(4)] Write down the \textit{critical portrait} of $f$.
         In our example it is: 
         \begin{center}
             \begin{tikzcd}
a \arrow[rd, "2"]  &                        &                                                   \\
-a \arrow[r, "2"]  & -1 \arrow[rd, "2"]     &                                                   \\
0 \arrow[ru, "2"]  &                        & 1 \arrow["2", loop, distance=2em, in=325, out=35] \\
b \arrow[r, "3"]   & \infty \arrow[ru, "2"] &                                                   \\
-b \arrow[ru, "3"] &                        &                                                  
\end{tikzcd}
         \end{center}
         where the numbers above the arrows indicate the multiplicity of the critical point.
         Note that in this example it happens that $\post(f) \subseteq \crit(f)$, but this is not always the case.\footnote{Here, following the labeling of \cite{Book}, we named some of the points as ``$0$'', ``$1$'', ``$-1$'', ``$\infty$''. 
         There, those labels refer to an identification of $S^2$ with $\C\mathbb{P}^1$, but for us they are just names of points.}

         \item [(5)] Identify $S^1$ with $\R/\Z$ by choosing coordinates for the leaves we obtained and the postcritical points so that all leaves are critical leaves and they satisfy the same critical portrait as $f$ under the map $h(t) = d\cdot t$.
        In our example, a choice could be the following, where all numbers should be divided by 36:
            \begin{center}
                \begin{tikzcd}
{\{3,33\}} \arrow[rd]    &                             &                                                            \\
{\{15,21\}} \arrow[r]    & {\{12,18\}} \arrow[rd] &                                                            \\
{\{9,27\}} \arrow[ru]    &                             & {\{0,30\}} \arrow[ loop, distance=2em, in=325, out=35] \\
{\{7,19,31\}} \arrow[r]  & {\{6,24\}} \arrow[ru]  &                                                            \\
{\{1,13,25\}} \arrow[ru] &                             &                                                           
\end{tikzcd}
            \end{center}

            \item [(6)] Using the combinatorics of step 3 and the numbers of step 5, draw a pair of degree $d$ majors.
            For this pair, the dynamics of the map $t \mapsto d\cdot t$ acting on the leaves are the same as the dynamics of the map $f$ acting on its critical points (figure \ref{1st approximation}, down left).

            \begin{figure}[h!]
             \centering
             \includegraphics[width=0.56\linewidth]{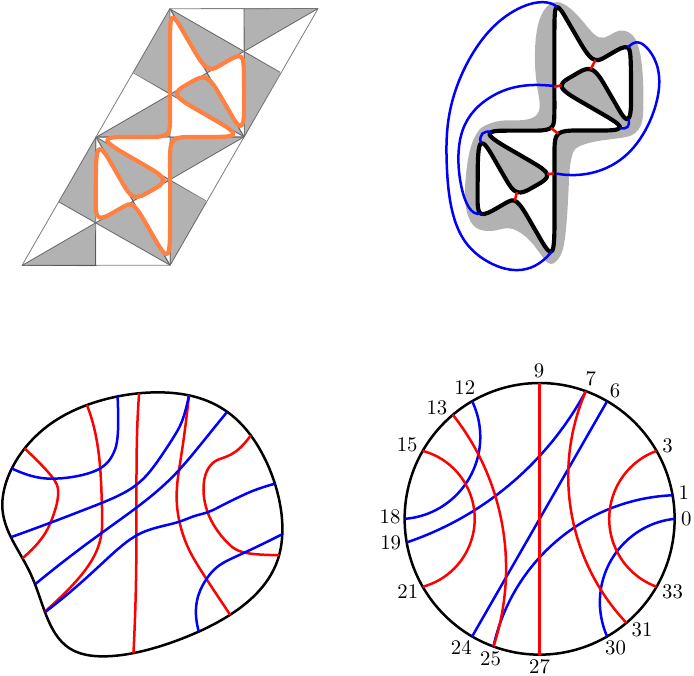}
             \caption{The steps described for obtaining a pair of degree 6 majors for the barycentric subdivision map.}
             \label{1st approximation}
         \end{figure}

    \end{enumerate}

    Under this construction, all the information of the first approximation $\gamma_1$ is enclosed in the information of the pair of majors.

    As Meyer showed in \cite{Quotients}, the data of the $n$-th approximation $\gamma_n$ is enclosed in the information of the $n$-th pullback of the pair of majors, $\Lambda^{\pm}_n$ in a similar way. In the limit, the information of the invariant sphere filling curve is given by the invariant laminations $\Lambda^{\pm}$. 

    \begin{figure}[h!]
        \centering
        \includegraphics[width=0.7\linewidth]{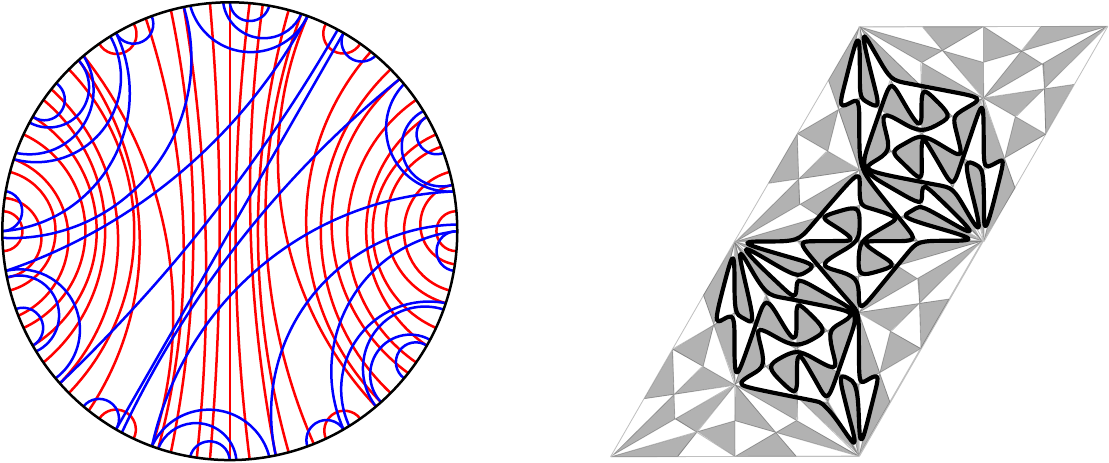}
        \caption{The laminations $\Lambda^{\pm}_2$ and the second approximation $\gamma_2$ to the invariant curve for $f$.}
        \label{2nd approximation}
    \end{figure}


\section{Decompositions for the invariant curve}\label{pullback}

In section \ref{mating} we described a process to get a pair of degree $d$ majors that precisely describe the first approximation $\gamma_1$ of an $f$-invariant sphere-filling curve.
In section \ref{pullback majors}, we saw that it is possible to pullback this pair of majors to create a sequence of pairs of finite laminations $\Lambda^{\pm}_n$ and a limiting pair of $(t \mapsto d \cdot t)$-invariant laminations $\Lambda^{\pm}$.
Meyer proves that the pullbacks $\Lambda^{\pm}_n$ precisely describe the approximation $\gamma_n$ in the same fashion. 
Moreover, he shows that since $f$ is combinatorially expanding, the approximations $\gamma_n$ converge to a continuous surjection $\gamma\colon S^1\to S^2$ with $f\circ\gamma=\gamma\circ h$ \cite{Quotients}, whose fibers are recorded by $\Lambda^\pm$; thus $S^2=\gamma(S^1)$ arises directly, not as an abstract  quotient.


This means that understanding the finite and easily computable data of $\Lambda^{\pm}_n$ we can get information about the (harder to compute) approximations $\gamma_n$. Recall that in section \ref{connections} we mentioned that the connection at a critical point of $f^2$ that is also a critical point of $f$ may be hard to compute. 
Using this correspondence and the construction of section \ref{connections and decompositions}, once we have the pair of laminations $\Lambda^{\pm}_2$, we can extract the decomposition element $[v]_2$ out of them and use it to draw the connection at $v$ for the second approximation. 
Then, we can follow along the circle $S^1$ and every time we cross an endpoint of a leaf we know that the approximation $\gamma_2$ has to pass through the corresponding critical point.
This is how figure \ref{2nd approximation} was created.


\subsection{Decomposition elements for pullbacks}\label{decompositions for pullbacks}

Since it is a simple numerical calculation, finding out the decomposition element $[v]_n$ can easily be done by a computer by just pulling back the majors $n$ times and seeing how the corresponding decomposition element looks like.

If all critical points of $f$ are preperiodic, their degree under iterates of $f$ is bounded, and thus
finitely many iterates are enough to know the connection at all vertices at all stages.
Things become more tricky though when some critical point $v$ is periodic.
Then, $v$ is a critical point of all iterates $f^n$ and its multiplicity grows, i.e. at each subdivision stage more and more white and black tiles appear around it.
This would require an infinite number of pullbacks to know the connection at $v$ at all stages.

Luckily, lemma \ref{lemma} helps us understand these decomposition elements given only the data of the pair of majors.


\begin{restatable}{lem}{MainLemma}\label{lemma}
    Suppose $a \in \crit(f)$ with $\deg_f(a) = k$ and $f(a) = b \in \crit(f)$. 
    Then $[a]_2$ consists exactly from $[a]_1$ joint with $k$ copies of $[b]_1$, each sharing an endpoint with a vertex of $[a]_1$.
\end{restatable}

\begin{proof}
    Let $x,y \in [a]_1$ be endpoints, and $z := h(x) = h(y)$. The point $z$ is well-defined since $[a]_1$ is a critical leaf and $z \in [b]_1$ since $f(a) = b$.
    Thus, there are leaves $\lambda,\mu \in h^{-1}([b]_1)$ with $x \in \lambda$ and $y \in \mu$.
    Note that $\lambda \neq \mu$ since otherwise $h([a]_1)$ would contain the leaf $h(\lambda) = [b]_1$ which is impossible.  
    Since the leaves $\lambda$ and $\mu$ are preimages of $[b]_1$ under $h$ which is a covering map, they are identical copies of $[b]_1$, and the lemma is proven.
\end{proof} 

\begin{corollary}
    The data of the pair of majors $C^{\pm}$ are enough to know the decomposition element $[v]_n$ for all critical points $v$ and all $n$.
\end{corollary}

\begin{proof}
    Apply lemma \ref{lemma} $n-1$ times.
\end{proof}

Note that, even though we know $[v]_n$ for all $n$, we still don't know exactly how $[v] \in \Lambda^{\pm}$ looks like.
It definitely contains $[v]_{\infty}:=\bigcup_n[v]_n$ but after taking closures there is the possibility that more points get identified with it.

Also note that at this stage it is possible that after some number of steps of adjoining new leaves to $[v]$, a chain of consecutive leaves comes back to itself creating a cycle, obstructing the existence of an infinite chain which is the main interest of this paper.
In lemma \ref{simply connected} though we will see that when $\Lambda^{\pm}$ are known to describe a sphere filling curve this is not possible.

\subsection{Laminations arising from a sphere-filling circle.}\label{decompositions}

It is now time to talk about decompositions arising from laminations that give rise to sphere-filling curves.

Suppose that a sphere-filling map $f:S^1\to S^2$ is given by collapsing a pair of laminations $\Lambda^{\pm}$.
The pair $\Lambda^{\pm}$ gives rise to an equivalence relation on $S^1$, that extends to an equivalence relation of a sphere $\mathcal{S}:=S^1 \sqcup P^{\pm}$, where $P^{\pm}$ are two copies of the plane. 
Thus we obtain a decomposition of the sphere $\mathcal{S}$ and $f$ extends to a quotient map $\hat{f}:\mathcal{S} \to S^2$ see \cite{CaTherineWheels}, section 2 for details.
Moore's theorem \cite{Moore} gives necessary conditions for a decomposition of the sphere $S^2$ to give a quotient homeomorphic to $S^2$.
Here we will only need a partial converse that we quickly prove in lemma \ref{simply connected}.

\begin{lemma}\label{simply connected}
    Suppose that $\mathcal{D}$ is a decomposition of the sphere $\mathcal{S}$, such that $\mathcal{S}/\mathcal{D}$ is homeomorphic to $S^2$. Then, every decomposition element $\xi$ of $\mathcal{D}$ is non-separating, i.e. $\mathcal{S} \smallsetminus \xi$ is connected.
\end{lemma}

\begin{proof}
    Suppose that there is an element $\xi \in \mathcal{D}$ that is separating, and let $P_i$ be the components of its complement.
    Let $\hat{f}:\mathcal{S} \to S^2 / \mathcal{D}$ be the quotient map.
    Then, $\mathcal{S}/\mathcal{D} \smallsetminus \hat{f}(\xi) = \bigsqcup_i \hat{f}(P_i) $ is disconnected.
    In particular, $\hat{f}(\xi)$ is a cut point and thus $\mathcal{S}/\mathcal{D}$ is not homeomorphic to $S^2$.
\end{proof}

Going back to the pair of laminations $\Lambda^{\pm}$, lemma $\ref{simply connected}$ implies the following:

\begin{corollary}
    The pair of laminations $\Lambda^{\pm}$ contains no cycles that consist of chains of perfect fits or limits of those.
\end{corollary}

\begin{proof}
    Indeed, such a chain would give rise to a decomposition element on the sphere $\mathcal{S}$ that is not simply connected and thus it is separating.
\end{proof}

Therefore, using lemmas \ref{lemma} and \ref{simply connected} we can now inductively combinatorially describe the decomposition elements $[v]_n$ for all $n$, where $v$ is a critical point of $f$, when the associated laminations describe approximations to a sphere filling curve.
As a consequence of this description, we can now prove corollary \ref{a and b}, which we recall from the introduction:

\CorAB*

\begin{proof}
    By lemma \ref{lemma}, the element $[a]_2$ contains a leaf of $\Lambda^+$ that is joint with leaves of $\Lambda^-$ (coming form $[b]_1$) on either side. 
    Inductively, we can see that $[a]_n$ contains alternating chains of length $\geq n$ and thus, since $[a]_n \subseteq [a]_{\infty}\subseteq [a]$ for all $n$, $[a]$ contains an infinite chain.
\end{proof}

Before discussing an example of an expanding Thurston map for which the conditions of corollary \ref{a and b} hold, we will conclude our running example by exploring how the decomposition elements for its critical points look like.
First, we re-draw the ramification portrait of section \ref{mating} like in figure \ref{portrait subdivision}:
\begin{figure}[h!]
    \centering
    \includegraphics[width=0.4\linewidth]{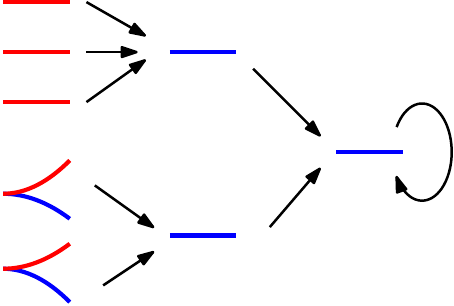}
    \caption{The ramification portrait of the barycentric subdivision, in terms of leaves.}
    \label{portrait subdivision}
\end{figure}

Using the process implied by lemma \ref{lemma} and lemma \ref{simply connected}, we can create the decomposition elements $[b]_1,[b]_2,[b]_3$, shown in figure \ref{3steps}.
\begin{figure}[h!]
    \centering
    \includegraphics[width=0.6\linewidth]{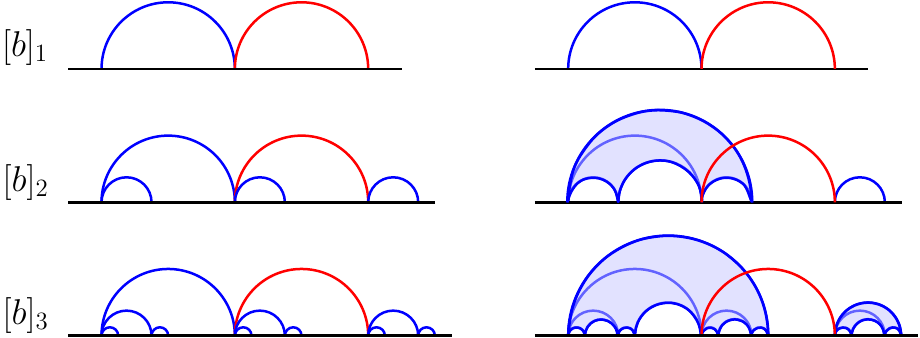}
    \caption{The decomposition elements $[b]_i$ for $i = 1,2,3$. (Left) in terms of leaves. (Right) In terms of ideal polygons.}
    \label{3steps}
\end{figure}

Applying the same process for all critical points of $f$ and translating the decomposition elements to connections as described in section \ref{connections and decompositions}, in figure \ref{laminations ramification} we show the connections at the corresponding vertices will look like for the third approximation for the sphere-filling curve for $f$, this time without pulling back the whole $\Lambda^{\pm}_2$.

\begin{figure}[h!]
    \centering
    \includegraphics[width=0.7\linewidth]{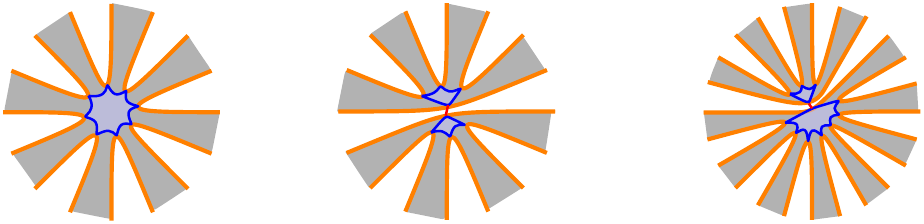}
    \caption{The connections at step 3 at the points (Left) $1,-1,\infty$, (Middle) $a,-a,0$ and (Right) $b,-b$.}
    \label{laminations ramification}
\end{figure}

Note that for this example we can easily see that, even though all critical points have eventually unbounded ramification, at stage $n$ the corresponding decomposition elements consist of two large blue polygons and at most one red leaf joining them.

\section{Infinite chains}\label{example}

We conclude by describing an expanding Thurston map for which some decomposition elements in the corresponding $\Lambda^{\pm}$ contain infinite chains of perfect fits and therefore give a proof of theorem \ref{maintheorem}.


\subsection{The example} 

Consider the branched covering of $S^2$ described in figure \ref{The map}. 

\begin{figure}[h!]
    \centering
    \includegraphics[width=0.6\linewidth]{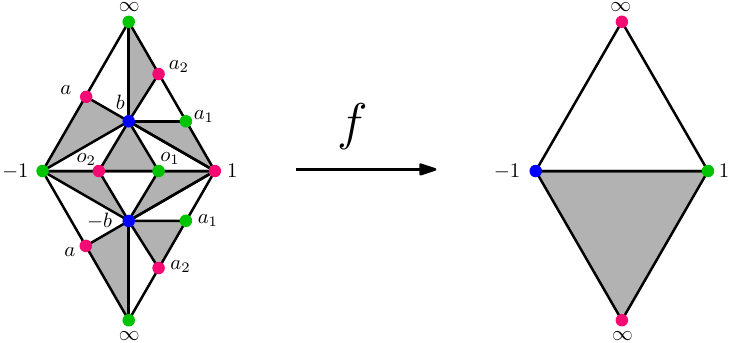}
    \caption{The map $f$.}
    \label{The map}
\end{figure}
It is a postcritically finite map $f:S^2 \to S^2$ of degree 8 and it is combinatorially expanding, as can be seen from figure \ref{2nd subdivision}. 

\begin{figure}[h!]
    \centering
    \includegraphics[width=0.18\linewidth]{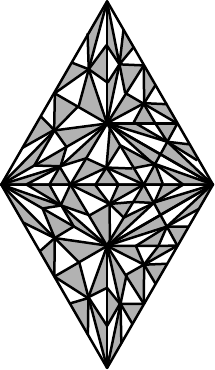}
    \caption{The second subdivision that comes from $f$. It is clear that the mesh size of the tiling goes to 0.}
    \label{2nd subdivision}
\end{figure}


The set of critical points of $f$ is: $\operatorname{crit}(f) = \{b,-b,a,a_1,a_2,o_1,o_2,1,-1,\infty\}$ and all of them have degree 2 except of $b,-b$ that have degree 4.
The set of postcritical points of $f$ is $\operatorname{post}(f) = \{1,-1,\infty\}$
\footnote{Here as in section \ref{mating}, the labels ``$1$'', ``$-1$'', ``$\infty$'' mean nothing more than labeling the points.}.
In particular, the ramification portrait of $f$ is the following:
\begin{center}
\begin{tikzcd}
b \arrow[rd, "4"]  &                   & a_1 \arrow[d, "2"']         &                                  & a \arrow[ld, "2"']   \\
                   & -1 \arrow[r, "2"] & 1 \arrow[r, "2", bend left] & \infty \arrow[l, "2", bend left] & o_2 \arrow[l, "2"']  \\
-b \arrow[ru, "4"] &                   & o_1 \arrow[u, "2"]          &                                  & a_2 \arrow[lu, "2"']
\end{tikzcd}
\end{center}

Since $|\operatorname{post}(f)| =3$, $f$ can be represented by a holomorphic map $g:\C \mathbb{P}^1 \to \C\mathbb{P}^1$, but this $g$ can't be chosen to be expanding for the same reason as the barycentric subdivision map, see \cite{Book} chapters 7 and 12 for details.


Consider also the Eulerian circuit $\gamma$ of figure \ref{The curve}.
\begin{figure}[h!]
    \centering
    \includegraphics[width=0.5\linewidth]{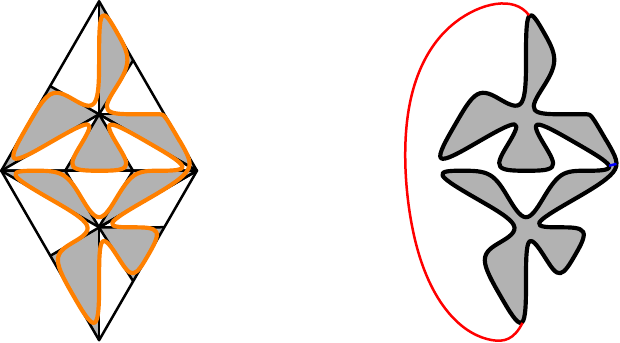}
    \caption{(Left) An Eulerian circuit for $f$. (Right) The leaves $[1]_1,[\infty]_1$.}
    \label{The curve}
\end{figure}
Note that for this choice of $\gamma$, the decomposition element $[1]_1$ is a single blue leaf, and $[\infty]_1$ is a single red leaf.
Therefore, focusing on this particular part of the ramification portrait, we get the portrait of figure \ref{1 and infty}.
\begin{figure}[h!]
    \centering
    \includegraphics[width=0.25\linewidth]{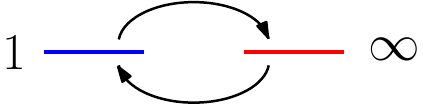}
    \caption{The critical points 1 and $\infty$ make up a cycle of length 2.}
    \label{1 and infty}
\end{figure}

Using corollary \ref{a and b}, it is straightforward that the decomposition element $[\infty] \in \Lambda^{\pm}$ contains an infinite chain of perfect fits. 
The decomposition elements $[\infty]_i$ for $i = 1,\dots,5$ and the corresponding connections at $\infty$ are shown in figure \ref{5 steps}.


\begin{figure}[h!]
    \centering
    \includegraphics[width=0.7\linewidth]{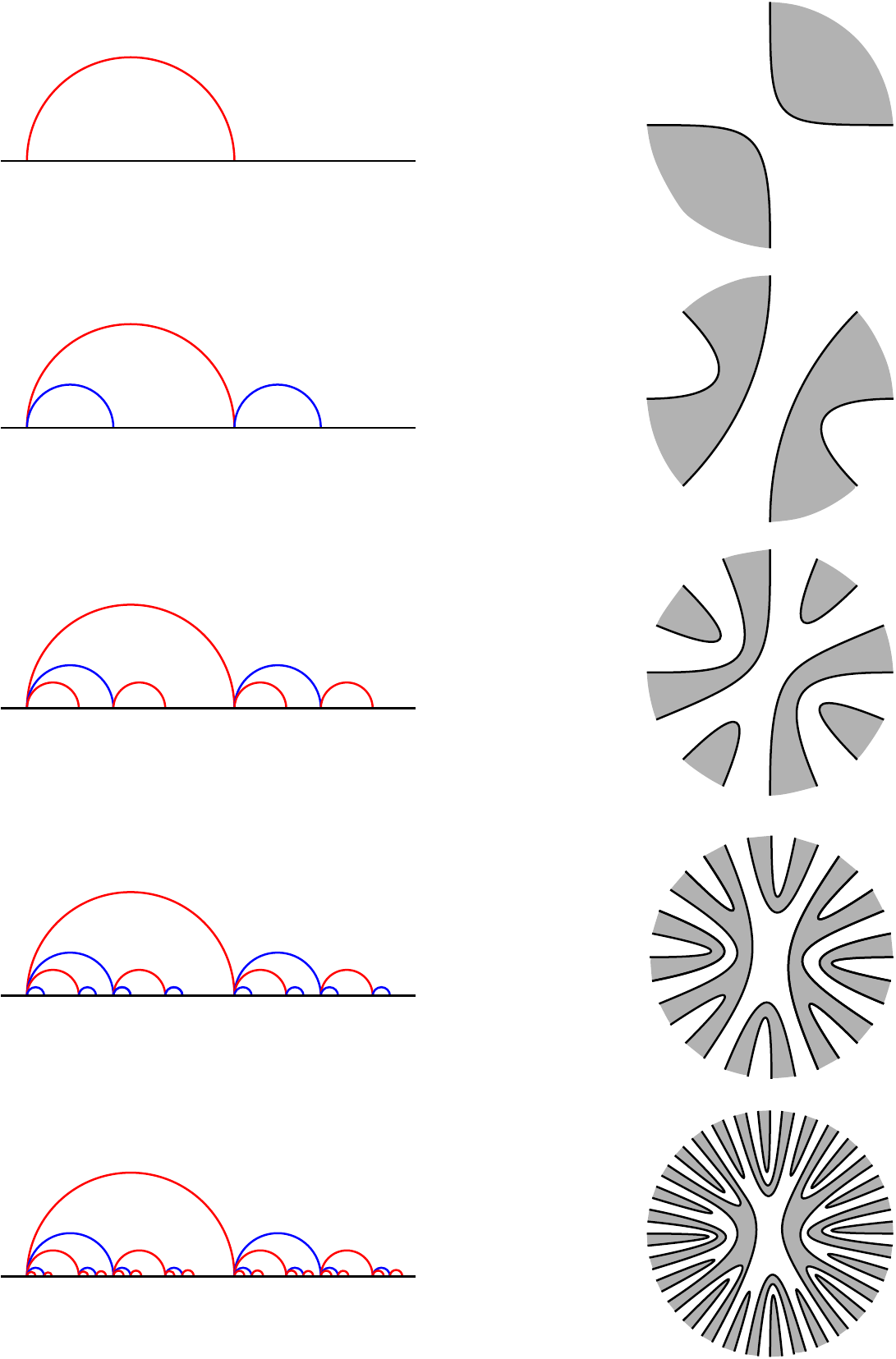}
    \caption{(Left) The decomposition elements $[1]_i, \ i =1,\dots,5$. (Right) The associated connections at $1$.}
    \label{5 steps}
\end{figure}

\bibliographystyle{alpha}
\bibliography{uhgeuir.bib}

\end{document}